\documentclass[11pt]{article}
\pdfoutput=1

\usepackage[final]{microtype}
\usepackage{amsthm, amsmath, amssymb}
\usepackage[pdfusetitle,pdfborder={0 0 .5}]{hyperref}
\usepackage{cleveref}

\newcommand\extrafootertext[1]{%
	\bgroup
	\renewcommand\thefootnote{\fnsymbol{footnote}}%
	\renewcommand\thempfootnote{\fnsymbol{mpfootnote}}%
	\footnotetext[0]{#1}%
	\egroup
}

\theoremstyle{plain}
\newtheorem{theorem}{Theorem}[section]
\newtheorem{lemma}[theorem]{Lemma}
\newtheorem{proposition}[theorem]{Proposition}

\theoremstyle{definition}
\newtheorem{definition}[theorem]{Definition}
\newtheorem{remark}[theorem]{Remark}
\newtheorem*{remark*}{Remark}
\newtheorem{example}[theorem]{Example}

\newcommand*{\ie}{{\it i.e.}}
\newcommand*{\eg}{{\it e.g.}}
\newcommand*{\R}{\mathbb{R}} 
\newcommand*{\Z}{\mathbb{Z}} 

\newcommand*{\norm}[1]{\left\lVert#1\right\rVert} 
\newcommand*{\bdddr}{{\widehat{H}}} 
\newcommand*{\bddforms}{{\widehat{\Omega}}} 

\title{Bounded differential forms and coinvariants of bounded functions}
\author{Francesco Milizia\\[-1pt]
	\footnotesize{\textit{Scuola Normale Superiore, Palazzo della Carovana,}}\\[-3pt]
	\footnotesize{\textit{Piazza dei Cavalieri, 7, 56126 Pisa, IT}}\\[-1pt]
	\footnotesize{\texttt{francesco.milizia@sns.it}}}
\date{}

\begin{document}
	\maketitle
	
	\begin{abstract}
		Given a group $G$ acting cocompactly on a smooth manifold $M$ by deck transformations, there is an integration map, defined recently by Kato, Kishimoto and Tsutaya, from the top-degree bounded de Rham cohomology of $M$ to the coinvariants $\ell^\infty(G)_G$.
		We generalize its definition and show that it is an isomorphism.
		In the presence of boundary, a relative version of bounded de Rham cohomology is considered.
	\end{abstract}
	
	\extrafootertext{The author has been supported by INdAM, GNSAGA project CUP E55F22000270001.}
	
	\section{Introduction}
	Let $M$ be a smooth oriented manifold of dimension $n$, endowed with a free, properly discontinuous and cocompact action of a group $G$ by diffeomorphisms.
	In \cite{KKT22}, Kato, Kishimoto and Tsutaya consider a linear map
	\[
		\bddforms^n(M) \to \ell^\infty(G)
	\]
	sending bounded differential $n$-forms on $M$ to bounded functions from $G$ to $\R$.
	This is done by fixing a fundamental domain for the action $G\curvearrowright M$ and integrating differential forms on the $G$-translates of the fundamental domain.
	Boundedness of differential forms is to be understood with respect to a $G$-invariant Riemannian metric on $M$, whose choice is not important since $G\curvearrowright M$ is cocompact.
	As a consequence of a Stokes-type theorem, they observe that if $M$ is connected and has no boundary then the map above induces a linear map
	\begin{equation}\label{eq:bdddrmap}
		\bdddr^n(M) \to \ell^\infty(G)_G
	\end{equation}
	from the top-degree bounded de Rham cohomology of $M$ to the module of coinvariants of $\ell^\infty(G)$ (see \Cref{sec:definition} for the definitions).
	They use this map to define the index of a bounded vector field on $M$, as an element of $\ell^\infty(G)_G$, and obtain a version of the Poincaré-Hopf theorem for bounded vector fields.
	They also notice that the map in \eqref{eq:bdddrmap} is surjective, and ask whether it is an isomorphism \cite[Conjecture 1.5]{KKT22}.
	To support this claim, they show that it holds when $G$ is finite and for $G = \Z$ acting on the real line by translations.
	Moreover, they remark that from \cite{BW1992} and \cite{AttieBlock1998} we know that $\bdddr^n(M)$ and $\ell^\infty(G)_G$ are both isomorphic to the uniformly finite homology $H_0^{\rm uf}(G)$, so we already know that some isomorphism should exist. However, they say that ``this isomorphism is not explicit as it is given by a zig-zag of several isomorphisms''.
	The purpose of this note is twofold:
	\begin{itemize}
		\item We discuss more carefully the definition of the map \eqref{eq:bdddrmap}, generalizing it for manifolds with boundary and showing its independence from the choice of a fundamental domain, an issue which is not addressed in \cite{KKT22};
		\item We answer the question of Kato, Kishimoto and Tsutaya showing that the map \eqref{eq:bdddrmap} is an isomorphism --- also in the case with boundary: see \Cref{thm:isomorphism}.
	\end{itemize}
	
	\paragraph{Acknowledgements.}
	I thank my PhD advisor Roberto Frigerio for having read the preliminary version of this document.
	
	\section{Definition of the integration map}\label{sec:definition}
	Throughout the article, $M$ is a smooth oriented manifold of dimension $n$, possibly with boundary.
	For the moment, we do not assume $M$ to be connected (and generally it won't be compact either).
	As in the introduction, we assume that $M$ is endowed with a left action by diffeomorphisms of a group $G$, and that this action is free, properly discontinuous and cocompact, so that the quotient $G\backslash M$ is a compact smooth manifold (not necessarily orientable).
	In this situation $G$ is always finitely generated, and is finite if and only if $M$ is compact.
	
	\subsection{Coinvariants of bounded functions}
	We denote by $\ell^\infty(G)$ the real vector space of bounded real-valued functions from $G$ to $\R$.
	It is endowed with the left action of $G$ given by the formula
	\begin{align*}
		(g \cdot f)(h) = f(hg) && \forall f \in \ell^\infty(G) && \forall g,h \in G.
	\end{align*}
	The module of coinvariants, denoted by $\ell^\infty(G)_G$, is the quotient of $\ell^\infty(G)$ by the linear subspace generated by the elements of the form $f - g\cdot f$, with $f\in\ell^\infty(G)$ and $g\in G$.
	
	\begin{example}
		If $G$ is finite, then $\ell^\infty(G)_G$ is one-dimensional, because the functions of the form $f - g\cdot f$ generate the subspace of $\ell^\infty(G)$ consisting of functions whose sum over $G$ is $0$.
		There are groups for which $\ell^\infty(G)_G = 0$; in fact, this is the case if and only if $G$ is \emph{not} amenable \cite[Theorem 3.1]{BW1992}.
		If $G$ is infinite and amenable, then $\ell^\infty(G)_G$ is infinite dimensional (see, \eg, \cite{BlankDiana2015} and \cite[Section 6.7]{Blank2015}).
	\end{example}
	
	\subsection{Bounded de Rham cohomology}\label{subsec:b_deRham}
	Bounded de Rham cohomology can be defined for any Riemannian manifold.
	The metric is needed in order to define what it means for a differential form to be bounded.
	In our setting we can pick any $G$-invariant Riemannian metric; by cocompactness of the action $G\curvearrowright M$, the choice is irrelevant.
	
	\begin{definition}
		Let $\omega \in \Omega^k(M)$ be a differential form.
		Its sup norm is defined as
		\[ \norm{\omega}_\infty = \sup_{x \in M}\left\{\sup_{v_1, \dots, v_k \in T^1_xM} \omega(x)(v_1, \dots, v_k)\right\} \in [0,+\infty],\]
		where $v_1,\dots,v_k$ vary in the unit sphere of the tangent space at $x \in M$.
		The form $\omega$ is \emph{bounded} if $\norm{\omega}_\infty < +\infty$.
		We denote by $\bddforms^k(M)$ the vector space of differential $k$-forms $\omega \in \Omega^k(M)$ that are bounded and whose exterior derivative $d\omega \in \Omega^{k+1}(M)$ is bounded.
	\end{definition}
	Notice that $\bddforms^*(M)$ is a subcomplex of $\Omega^*(M)$, since the exterior differential sends $\bddforms^k(M)$ to $\bddforms^{k+1}(M)$.
	\begin{definition}
		The cohomology of $\bddforms^*(M)$, denoted by $\bdddr^*(M)$, is the \emph{bounded de Rham cohomology} of $M$.
	\end{definition}
	Bounded de Rham cohomology is studied, \eg, in \cite{AttieBlock1998} and \cite{Mil2021}, where it is shown to be isomorphic to $H^*(G\backslash M;\ell^\infty(G))$.
	We also consider a version relative to the boundary.
	\begin{definition}
		Let $\bddforms^k(M,\partial M) \subseteq \bddforms^k(M)$ be the subspace consisting of forms $\omega\in\bddforms^k(M)$ whose pull-back to $\partial M$ is $0$.
		This defines a further subcomplex of $\Omega^*(M)$, and we denote its cohomology by $\bdddr^*(M,\partial M)$.
	\end{definition}
	
	\subsection{Integration of bounded forms}
	Recall that we denote by $n$ the dimension of $M$, so that any $n$-form on $M$ vanishes on $\partial M$.
	We now define a linear map
	\[
		\bddforms^n(M,\partial M) = \bddforms^n(M) \to \ell^\infty(G).
	\]
	In \cite{KKT22}, the auxiliary device used to define such a map is a fundamental domain for $G\curvearrowright M$.
	Here, we take a more general perspective and consider, instead, a function in $L^1(M)$.
	\begin{definition}\label{def:intphi}
		For every $\varphi \in L^1(M)$, we define
		\[
			\int_M^\varphi: \bddforms^n(M) \to \ell^\infty(G)
		\]
		setting, for every $\omega \in \bddforms^n(M)$ and $g \in G$,
		\[
			\left(\int_M^\varphi\omega\right)(g) = \int_M \varphi \cdot g^*\omega.
		\]
	\end{definition}
	To recover the definition in \cite{KKT22}, one can take $\varphi$ to be the characteristic function of a fundamental domain.
	
	\begin{remark}\label{rm:L1action}
		The group $G$ acts on $L^1(M)$ in the usual way: $g \cdot \varphi (x) = \varphi(g^{-1}x)$.
		For every $\omega \in \bddforms^n(M)$, it holds
		\begin{equation}\label{eq:L1action}
			\left(\int_M^{g\cdot \varphi}\omega\right) = g \cdot \left(\int_M^\varphi\omega\right) \in \ell^\infty(G).
		\end{equation}
	\end{remark}
	
	\begin{lemma}\label{lemma:phiindep}
		Let $\varphi$ be a function in $L^1(M)$ with compact support such that, for almost every $x \in M$,
		\begin{align*}
			\sum_{g \in G} \varphi(gx) = 0.
		\end{align*}
		Then, for every $\omega \in \bddforms^n(M)$, $\int_M^\varphi\omega$ is trivial in $\ell^\infty(G)_G$.
	\end{lemma}
	\begin{proof}
		Fix a bounded fundamental domain $D \subseteq M$, \ie, a bounded Borel subset of $M$ such that its $G$-translates give a partition of $M$.
		Note that, since $D$ is bounded and the action $G \curvearrowleft M$ is properly discontinuous, the partition is locally finite.
		Since $\varphi$ has compact support, it (essentially) vanishes on all but a finite number of distinct translates $g_1D, \dots, g_kD$.
		For every $i \in \{1,\dots,k\}$ consider the following function:
		\[\varphi_i(x) = \begin{cases}
			\varphi(x) & x \in g_iD\\
			0 & x \not\in g_iD,
		\end{cases}\]
		which is in $L^1(M)$ and has compact support.
		Consider now
		\[\psi = \varphi - \sum_{i=1}^k (\varphi_i - g_i^{-1}\cdot\varphi_i).\]
		It follows from \Cref{rm:L1action} and linearity of integration that, for every $\omega\in\bddforms^n(M)$, $\int_M^\varphi\omega$ and $\int_M^\psi\omega$ give the same element in $\ell^\infty(G)_G$.
		Moreover, by construction, $\psi$ is supported in $D$ and has sum $0$ along almost every $G$-orbit, but this implies that $\psi$ vanishes almost everywhere and the conclusion follows.
	\end{proof}
	
	\begin{definition}
		Let $\varphi \in L^1(M)$ be compactly supported, satisfying
		\[
		\sum_{g \in G} \varphi(gx) = 1
		\]
		for almost every $x \in M$.
		We define 
		\[ \int_M: \bddforms^n(M) \to \ell^\infty(G)_G \]
		as the composition of the map $\int_M^\varphi$ with the projection $\ell^\infty(G) \to \ell^\infty(G)_G$.
		By \Cref{lemma:phiindep}, the map $\int_M$ does not depend on $\varphi$.
	\end{definition}
	
	If $M$ has boundary, by restriction we have an action $G\curvearrowright\partial M$.
	We recall the following Stokes-type result from \cite{KKT22}, where it is stated and proved for integration on $G$-translates of a specific fundamental domain obtained from a triangulation of $G\backslash M$, which is just (thanks to \Cref{lemma:phiindep}) a specific implementation of the map $\int_M$ defined above.
	\begin{proposition}\label{prop:stokes}
		For every $\omega \in \bddforms^{n-1}(M)$, it holds
		\[ \int_M d\omega = \int_{\partial M} i^*\omega \in \ell^\infty(G)_G,\]
		where $i:\partial M \to M$ is the inclusion.
	\end{proposition}
	\Cref{prop:stokes} readily implies that $\int_M$ induces a map on bounded relative de Rham cohomology:
	\begin{equation}\label{eq:intmap}
		\int_M: \bdddr^n(M,\partial M) \to \ell^\infty(G)_G.
	\end{equation}
	
	\section{The integration map is an isomorphism}
	In this section we prove the following.
	\begin{theorem}\label{thm:isomorphism}
		Let $M$ be a smooth oriented connected manifold of dimension $M$, possibly with boundary, endowed with a free, properly discontinuous and cocompact action by diffeomorphisms of a group $G$.
		Then the integration map \eqref{eq:intmap} is an isomorphism of real vector spaces.
	\end{theorem}
	Before going through the proof, we need to establish suitable versions of the Poincaré lemma, in which we control the norm of the primitive of a given differential form.
	The construction of the primitive is the same as in classical proofs of the Poincaré lemma (\eg, \cite[Theorem 3.2]{GH2004}); we just remark how we can bound the norm of the resulting primitive.
	This is by no means the first time that a ``bounded'' Poincaré lemma is considered; the numerous versions existing in the literature all consider slightly different settings.
	
	\subsection{Poincaré lemmas with control on the norm}
	We fix the following notation:
	\begin{itemize}
		\item $Q = (0,1)^n \subset \R^n$;
		\item $Q' = (0,1)^{n-1} \times [0,1) \subset \R^n$.
	\end{itemize}
	
	\begin{lemma}\label{lemma:poincare}
		Let $\omega \in \Omega^n(Q)$ be a smooth differential form with compact support and $\int_Q\omega = 0$.
		Then, there exists $\eta \in \Omega^{n-1}(Q)$ with compact support such that:
		\begin{itemize}
			\item $d\eta = \omega$;
			\item $\norm{\eta}_\infty \le K_n \cdot \norm{\omega}_\infty$, where $K_n \in \R$ does not depend on $\omega$.
		\end{itemize}
	\end{lemma}
	\begin{proof}
		Write $\omega = f dx_1 \wedge \dots \wedge dx_n$ and suppose that, for a certain $k \in \{1, \dots, n\}$, the following condition holds:
		\begin{align}\label{eq:fcond}
			\int_{x_n}\dots\int_{x_k} f(x_1,\dots,x_n)\ dx_k\cdots dx_n = 0 && 
			\forall (x_1, \dots, x_{k-1}) \in (0,1)^{k-1}.
		\end{align}
		By hypothesis, this is certainly true for $k = 1$ (while for $k = n$ the condition implies that the form vanishes).
		In the following, we show how to find $\eta_k \in \Omega^{n-1}(Q)$ with compact support such that $\omega - d\eta_k$ satisfies the condition with $k$ replaced by $k+1$.
		Repeating this up to $k = n$, and defining $\eta$ as the sum of the $\eta_k$'s, we get $\omega - d\eta = 0$, with the required control of the norm (as long as we control the norms of the $\eta_k$'s).
		Define $g:(0,1)^k\to \R$ as
		\[ g(x_1, \dots, x_k) = \int_{x_n}\dots\int_{x_{k+1}} f(x_1,\dots,x_n)\ dx_{k+1}\cdots dx_n.\]
		Notice that $g$ has compact support and $\norm{g}_\infty \le \norm{f}_\infty$.
		Now, define $h:(0,1)^k\to \R$ as
		\[ h(x_1, \dots, x_k) = \int_0^{x_k} g(x_1, \dots, x_{k-1},s)\ ds. \]
		Condition \eqref{eq:fcond} implies that $h$ has compact support, and again $\norm{h}_\infty \le \norm{f}_\infty$.
		Let $\rho:(0,1)^{n-k}\to \R$ be a smooth function with compact support and integral $1$. Finally, define
		\[ \nu_k = (-1)^{k-1} h(x_1,\dots,x_k)\rho(x_{k+1},\dots,x_n)\  dx_1\wedge\dots\wedge\widehat{dx_k}\wedge\dots\wedge dx_n.\]
		The form $\nu_k$ has compact support, satisfies $\norm{\nu_k}_\infty \le \norm{h}_\infty \cdot \norm{\rho}_\infty$, and a straightforward check proves that $\omega-d\nu_k$ satisfies \eqref{eq:fcond} with $k$ replaced by $k+1$.
	\end{proof}
	
	\begin{lemma}\label{lemma:poincareb}
		Let $\omega \in \Omega^n(Q')$ be a smooth differential form with compact support and $\int_{Q'}\omega = 0$.
		Then, there exists $\eta \in \Omega^{n-1}(Q')$ with compact support such that:
		\begin{itemize}
			\item $d\eta = \omega$;
			\item The pull-back of $\eta$ to the boundary $\partial Q' = (0,1)^{n-1}\times\{0\}$ vanishes;
			\item $\norm{\eta}_\infty \le K_n \cdot \norm{\omega}_\infty$, where $K_n \in \R$ does not depend on $\omega$.
		\end{itemize}
	\end{lemma}
	\begin{proof}
		Proceed as in \Cref{lemma:poincare}.
		We just have to check that, in the last step with $k = n$, the pull-back of $\eta_n = (-1)^{n-1} h(x_1,\dots,x_n)\ dx_1\wedge\dots\wedge dx_{n-1}$ vanishes.
		But this is true because $h(x_1,\dots,x_n) = 0$ whenever $x_n = 0$.
	\end{proof}
	
	\subsection{Injectivity of the integration map}
	Recall that we already know from \cite{KKT22} that the integration map is surjective; hence, we are left to show that it is also injective.
	
	We keep the notation $Q = (0,1)^n$ and $Q' = (0,1)^{n-1}\times [0,1)$.
	We fix a finite collection of open subsets $A_1, \dots, A_k \subset G\backslash M$, each endowed with an orientation-preserving diffeomorphism $A_i \cong Q$ or $A_i \cong Q'$, such that:
	\begin{itemize}
		\item Each $A_i$ trivializes the covering map $\pi: M \to G\backslash M$;
		\item $A_1 \cup \dots \cup A_k = G\backslash M$;
		\item The subsets $B_i \subset A_i$ corresponding, via the diffeomorphism $A_i \cong Q$ or $A_i \cong Q'$, to $(1/3,2/3)^n \subset \R^n$, are pairwise disjoint.
	\end{itemize}
	Then we fix smooth maps $\lambda_i : G\backslash M \to \R$ compactly supported in $A_i$, with $\lambda_1 + \dots + \lambda_k \equiv 1$ on $G\backslash M$.
	Note that $\lambda_i$ is constantly equal to $1$ in $B_i$.
	
	Now we lift everything to $M$ via the projection $\pi:M \to G\backslash M$.
	For each $i \in \{1,\dots,k\}$, we fix a lift $C_i$ of $A_i$ in $M$, \ie, a connected component of its preimage $\pi^{-1}(A_i) \subset M$.
	We denote by $D_i = C_i \cap \pi^{-1}(B_i)$ the lift of $B_i$ sitting inside $C_i$.
	
	Define $\varphi_i:G\backslash M \to \R$ to be identically $0$ outside $C_i$ and equal to the composition $\lambda\circ\pi$ on $C_i$.
	It is smooth and it is constantly equal to $1$ in $D_i$.
	Finally, set $\varphi = \varphi_1 + \dots + \varphi_k$.
	Henceforth, we use this specific $\varphi$ to define the integration map $\int_M^\varphi:\bddforms^n(M)\to\ell^\infty(G)$.
	
	\begin{lemma}\label{lemma:zerof}
		Let $\omega \in \bddforms^n(M)$ and suppose that $\int_M\omega = 0 \in \ell^\infty(G)_G$.
		Then there exists $\eta \in \bddforms^{n-1}(M,\partial M)$ such that $\int_M^\varphi(\omega-d\eta) = 0 \in \ell^\infty(G)$.
	\end{lemma}
	\begin{proof}
		We have by assumption that
		\[ \int_M^\varphi\omega = \sum_{j=1}^r(f_j - g_j\cdot f_j) \in \ell^\infty(G) \]
		for some $f_1, \dots, f_r \in \ell^\infty(G)$ and $g_1,\dots,g_r \in G$.
		More explicitly, this yields
		\[ \int_M \varphi \cdot h^*\omega = \sum_{j=1}^r(f_j(h)-f_j(h g_j)) \in \R \]
		for every $h \in G$.
		For $j \in \{1,\dots,r\}$, fix a smooth orientation-preserving embedding $\theta_j:Q \to M$ that ``connects'' $g_j^{-1}D_1$ to $D_1$ in the following sense:
		\begin{itemize}
			\item $\theta_j(x_1,\dots,x_n) \in g_j^{-1}D_1$ for $x_n < 1/3$ (the ``lower part'' of $Q$);
			\item $\theta_j(x_1,\dots,x_n) \in D_1$ for $x_n > 2/3$ (the ``upper part'' of $Q$).
		\end{itemize}
		Fix $\rho \in \Omega^n(Q)$ with compact support, with integral $1$ in the upper part of $Q$, integral $-1$ in the lower part of $Q$, and vanishing in the ``central part'' ($1/3 \le x_n \le 2/3$).
		By \Cref{lemma:poincare}, $\rho=d\nu$ for a certain $\nu \in \Omega^{n-1}(Q)$ with compact support.
		We think of $\theta_j^*\rho$ and $\theta_j^*\nu$ as forms on $M$, by extending them to $0$ away from $\theta_j(Q)$.
		The key property of $\theta_j^*\rho$ is that
		\[ \int_M \varphi \cdot h^*(g^{-1})^*\theta_j^*\rho = \begin{cases}
			1 & h = g\\
			-1 & h = gg_j^{-1}\\
			0 & \text{else}
		\end{cases}\]
		for every $h,g \in G$.
		We define
		\[ \eta = \sum_{j=1}^r\sum_{g \in G} f_j(g)\ (g^{-1})^*\theta_j^*\nu. \]
		Since $\theta_j^*\nu$ has compact support and $G\curvearrowright M$ is properly discontinuous, on every compact set only a finite number of terms do not vanish, hence the sum is well defined.
		Moreover, $\eta \in \bddforms^{n-1}(M,\partial M)$.
		For every $h \in G$, we have
		\begin{align*}
			\int_M \varphi \cdot h^*d\eta &= \sum_{j=1}^r\sum_{g \in G} f_j(g) \int_M \varphi \cdot h^*(g^{-1})^*\theta_j^*\rho\\
			&= \sum_{j=1}^r (f_j(h) - f_j(hg_j)) = \int_M \varphi\cdot h^*\omega,
		\end{align*}
		that is, $\int_M^\varphi(\omega-d\eta) = 0$.
	\end{proof}
	
	\begin{lemma}\label{lemma:zeroint}
		Let $\omega \in \bddforms^n(M)$ and suppose that $\int_M\omega = 0 \in \ell^\infty(G)_G$.
		Then there exists $\eta \in \bddforms^{n-1}(M,\partial M)$ such that
		\[
			\int_M \varphi_i \cdot g^*(\omega-d\eta) = 0
		\]
		for every $g \in G$ and $i \in \{1,\dots,k\}$.
	\end{lemma}
	\begin{proof}
		By \Cref{lemma:zerof} we can assume that
		\[ \sum_{i=1}^k\int_M \varphi_i\cdot g^*\omega =  \int_M \varphi \cdot g^*\omega = 0\]
		for every $g \in G$.
		As in the proof of \Cref{lemma:zerof}, we use suitable differential forms in $\bddforms^{n-1}(M,\partial M)$ --- which are pull-backs of $\rho$ --- to ``move mass'' between the subsets $g \cdot D_i$.
		This time, the exchanges are among pairs of such subsets sharing the same $g$, but with a different index $i$.
	\end{proof}
	
	\begin{proof}[Proof of \Cref{thm:isomorphism}]
		We can now conclude the proof that the integration map is injective, hence bijective.
		Let $\omega \in \bddforms^n(M)$ and suppose that $\int_M\omega = 0 \in \ell^\infty(G)_G$.
		We have to find $\eta \in \bddforms^{n-1}(M,\partial M)$ such that $d\eta = \omega$.
		By \Cref{lemma:zeroint} we can assume that
		\[ \int_M \varphi_i \cdot g^*\omega = 0 \]
		for every $g \in G$ and $i \in \{1,\dots,k\}$.
		
		The differential form $\varphi_i \cdot g^*\omega$ is compactly supported in $C_i \subset M$.
		By the suitable version of the Poincaré lemma --- \Cref{lemma:poincare} or \Cref{lemma:poincareb} according to whether $C_i$ has boundary or not --- there is $\eta_i^g \in \Omega^{n-1}(M,\partial M)$ which is compactly supported in $C_i$, with $\norm{\eta}_\infty \le K\cdot\norm{\omega}_\infty$ for some constant $K$ (which depends on the constant $K_n$ in the lemma and on the fixed diffeomorphism between $A_i$ and $Q$ or $Q'$), and $d\eta_i^g = \varphi_i \cdot g^*\omega$.
		We conclude by defining
		\[ \eta = \sum_{i=1}^k\sum_{g \in G} (g^{-1})^*\eta_i^g,\]
		whose exterior differential is
		\[ d\eta = \sum_{i=1}^k\sum_{g \in G} (g^{-1})^*(\varphi_i\cdot g^*\omega) = \sum_{i=1}^k\sum_{g \in G} (g\cdot\varphi_i) \cdot\omega, \]
		which is equal to $\omega$ because the functions $g\cdot \varphi_i$ give a (uniformly locally finite) partition of unity on $M$.
	\end{proof}
	
	\bibliography{bibliography}

\providecommand{\bysame}{\leavevmode\hbox to3em{\hrulefill}\thinspace}
\providecommand{\MR}{\relax\ifhmode\unskip\space\fi MR }
\providecommand{\MRhref}[2]{%
  \href{http://www.ams.org/mathscinet-getitem?mr=#1}{#2}
}
\providecommand{\href}[2]{#2}
\begin{thebibliography}{1}

\bibitem{AttieBlock1998}
O.~Attie and J.~Block, \emph{Poincar{\'e} duality for {$L^p$} cohomology and
  characteristic classes}, Tech. Report 98-48, DIMACS, 1998.

\bibitem{Blank2015}
M.~Blank, ``Relative {B}ounded {C}ohomology for {G}roupoids'', Ph.D. thesis,
  University of Regensburg, 2015.
  \href{https://doi.org/10.5283/epub.31298}{\path{doi:10.5283/epub.31298}}

\bibitem{BlankDiana2015}
M.~Blank and F.~Diana, \emph{Uniformly finite homology and amenable groups},
  Algebr. Geom. Topol. \textbf{15} (2015), 467--492.
  \href{https://doi.org/10.2140/agt.2015.15.467}{\path{doi:10.2140/agt.2015.15.467}}

\bibitem{BW1992}
J.~Block and S.~Weinberger, \emph{Aperiodic tilings, positive scalar curvature
  and amenability of spaces}, J. Amer. Math. Soc. \textbf{5} (1992), 907--918.
  \href{https://doi.org/10.2307/2152713}{\path{doi:10.2307/2152713}}

\bibitem{GH2004}
V.~Guillemin and P.~Haine, ``Differential Forms'', World Scientific Publishing
  Co. Pte. Ltd., Hackensack, NJ, 2019.
  \href{https://doi.org/10.1142/11058}{\path{doi:10.1142/11058}}

\bibitem{KKT22}
T.~Kato, D.~Kishimoto and M.~Tsutaya, \emph{Vector fields on non-compact
  manifolds}, 2022. Accepted for publication in Algebr. Geom. Topol.
  \href{https://doi.org/10.48550/arXiv.2211.00512}{\path{doi:10.48550/arXiv.2211.00512}}

\bibitem{Mil2021}
F.~Milizia, \emph{$\ell^\infty$-cohomology: amenability, relative
  hyperbolicity, isoperimetric inequalities and undecidability}, 2021.
  \href{https://doi.org/10.48550/arXiv.2107.09089}{\path{doi:10.48550/arXiv.2107.09089}}

\end{thebibliography}
\end{document}